\documentclass[11pt,reqno]{amsart}

\usepackage{amsfonts, amsthm, amsmath}
\allowdisplaybreaks[4]

\usepackage{caption}

\usepackage{rotating}

\usepackage{tikz}

\usepackage{graphics}

\usepackage{amssymb}

\usepackage{amscd}

\usepackage[latin2]{inputenc}

\usepackage{t1enc}

\usepackage[mathscr]{eucal}

\usepackage{indentfirst}

\usepackage{graphicx}

\usepackage{graphics}

\usepackage{pict2e}

\usepackage{mathrsfs}

\usepackage{enumerate}
\usepackage[pagebackref]{hyperref}
\hypersetup{colorlinks=true}
\usepackage{cite}
\usepackage{color}
\usepackage{epic}
\usepackage{hyperref} 
\usepackage{framed}
\usepackage{mathabx}

\usepackage{ulem}

\numberwithin{equation}{section}
\theoremstyle{plain}

\newtheorem{theorem}{Theorem}[section]

\newtheorem{lemma}[theorem]{Lemma}

\newtheorem{proposition}[theorem]{Proposition}

\theoremstyle{definition}

\newtheorem{example}[theorem]{Example}

\newtheorem{?}[theorem]{Problem}

\newtheoremstyle{named}{}{}{\itshape}{}{\bfseries}{.}{.5em}{#1\thmnote{ #3}}
\theoremstyle{named}

\newcommand{\abs}[1]{\left|#1\right|}
\newcommand{\set}[1]{\left\{#1\right\}}

\newcommand{\f}[1]{\ifthenelse{\equal{#1}{1}}{(q;q)_\infty}{(q^{#1};q^{#1})_{\infty}}}

\def\bR{\mathbb{R}}
\def\bP{\mathbb{P}}
\def\PD{\mathrm{PD}}
\def\PDOset{\mathcal{PDO}}
\def\PDO{\mathrm{PDO}}
\def\ld{\ell_d}
\def\ldo{\ell_d^o}
\def\lr{\ell_r}
\def\cH{\mathcal{H}}

\begin{document}

\title[A PDO refinement]{A refined view of a curious identity for partitions into odd parts with designated summands}

\author[S. Fu]{Shishuo Fu$^{\ast}$}
\address[Shishuo Fu]{College of Mathematics and Statistics, Chongqing University \& Key Laboratory of Nonlinear Analysis and its Applications (Chongqing University), Ministry of Education, Chongqing 401331, P.R. China}
\email{fsshuo@cqu.edu.cn}
\thanks{$^{\ast}$Corresponding author.}

\author[J. A. Sellers]{James A. Sellers}
\address[James A. Sellers]{Department of Mathematics and Statistics, University of Minnesota Duluth, Duluth, MN 55812, USA}
\email{jsellers@d.umn.edu}

\date{\today}

\begin{abstract}

In 2002, Andrews, Lewis, and Lovejoy introduced the combinatorial objects which they called {\it partitions with designated summands}.  These are constructed by taking unrestricted integer partitions and designating exactly one of each occurrence of a part.  In the same work, they also considered the restricted partitions with designated summands wherein all parts must be odd, and they denoted the corresponding function by $\PDO(n)$.  

One of the most curious identities satisfied by $\PDO(n)$ is the following:  
 For all $n\geq 0$,
\begin{equation*}
\PDO(2n)=\sum_{k=0}^{n}\PDO(k)\PDO(n-k).
\end{equation*}
From a generating function perspective, this is equivalent to proving that 
\begin{equation*}
\sum_{n\geq 0} \PDO(2n)q^n = \left( \sum_{n\geq 0} \PDO(n)q^n \right)^2. 
\end{equation*}
The above results are easily proven via elementary generating function manipulations.  

In this work, our goal is to extend the above results by providing a two--parameter generalization of the above generating function identity.  We do so by considering various natural partition statistics and then utilizing ideas that appear in a 2013 paper of Andrews and Rose which date back a century to groundbreaking work of P. A. MacMahon on natural generalizations of divisors sums.  

\end{abstract}

\keywords{Integer partition, partition with designated summands, partition identity, MacMahon's generalized divisors sum, Chebyshev polynomial.
\newline \indent 2020 {\it Mathematics Subject Classification}. 11P84, 05A17, 05A15, 33B10.}

\maketitle

\section{Introduction and Motivation}\label{sec:intro}


In 2002, Andrews, Lewis, and Lovejoy \cite{ALL02} introduced the combinatorial objects which they called {\it partitions with designated summands}.  These are constructed by taking unrestricted integer partitions and designating exactly one of each occurrence of a part.  For example, from the five partitions of weight $4$ which are given by 
\begin{gather*}
    4,\ \ \  3+1,\ \ \  2+2,\ \ \   2+1+1,\ \ \   1+1+1+1,
\end{gather*}
we have ten partitions with designated summands of weight $4$:  
\begin{gather*}
    4',\ \ \  3'+1',\ \ \  2'+2,\ \ \  2+2',\ \ \  2'+1'+1,\ \ \  2'+1+1',\\
    1'+1+1+1,\ \ \  1+1'+1+1,\ \ \  1+1+1'+1,\ \ \  1+1+1+1'.
\end{gather*}
The authors denoted the number of partitions with designated summands of weight $n$ by $\PD(n)$.  Using this notation and the example above, we see that $\PD(4)=10$.

In \cite{ALL02}, Andrews, Lewis, and Lovejoy also considered the restricted partitions with designated summands wherein all parts must be odd, and they denoted the corresponding function by $\PDO(n)$.  Thus, from the example above, we see that $\PDO(4)=5$ given the following five objects:  
$$
3'+1',\ \ \  1'+1+1+1,\ \ \  1+1'+1+1,\ \ \  1+1+1'+1,\ \ \  1+1+1+1'.
$$
Andrews, Lewis, and Lovejoy noted \cite[(1.6)]{ALL02} that the generating function for $\PDO(n)$ is given by 
\begin{equation}
    \label{genfn_main}
    \sum_{n\geq 0} \PDO(n)q^n = \frac{f_4f_6^2}{f_1f_3f_{12}},
\end{equation}
where $f_r = (1-q^r)(1-q^{2r})(1-q^{3r})(1-q^{4r})\dots$ is the usual $q$--Pochhammer symbol, and we shall use the more standard notation $(q^a;q^b)_{\infty}=(1-q^a)(1-q^{a+b})(1-q^{a+2b})\cdots$ interchangeably.

Thanks to (\ref{genfn_main}), one can easily find the values of $\PDO(n)$ for small $n$.  Table \ref{table-PDO-values} provides the first several values of $\PDO(n)$.  

\begin{center}
\begin{table}
\begin{tabular}{ |c| *{11}{r|} }
\hline 
$n$  & 0 & 1 & 2 & 3 & 4 & 5 & 6 & 7 & 8 & 9 & 10 \\
\hline
$\PDO(n)$ & 1 & 1 & 2 & 4 & 5 & 8 & 12 & 16 & 22 & 32 & 42 \\
\hline 
\end{tabular}
\caption{Values of $\PDO(n)$} \label{table-PDO-values}
\end{table}
\end{center}

The primary focus of much of the recent research on the functions $\PD(n)$ and $\PDO(n)$ has been arithmetic, with an emphasis on divisibility properties satisfied by the two functions.  
(See \cite{ALL02, BO15, CJJS, CS24, Her23, HBN17, sel24, Xia16} for examples of such work.)  Our focus in this note is much different.  Namely, we wish you consider the following curious identity satisfied by $\PDO(n)$.

\begin{theorem}
\label{id:ori-pdo}  
 For all $n\geq 0$,
\begin{equation*}
\PDO(2n)=\sum_{k=0}^{n}\PDO(k)\PDO(n-k).
\end{equation*}
From a generating function perspective, this is equivalent to proving that 
\begin{equation}
\label{genfn:ori-pdo}
\sum_{n\geq 0} \PDO(2n)q^n = \left( \sum_{n\geq 0} \PDO(n)q^n \right)^2. 
\end{equation}
\end{theorem}
This result has appeared in the past literature.  It is implicitly apparent in \cite{ALL02, sel24}, and it is explicitly called out in the final section of the paper of Baruah and Ojah \cite{BO15} where the authors note that it would be  ``interesting to find a combinatorial proof of this identity''.  In all of the above--mentioned works,  (\ref{genfn:ori-pdo}) was proved via elementary 2--dissection of the generating function given in (\ref{genfn_main}).   

\begin{example}
From Table \ref{table-PDO-values}, we see that $\PDO(8) = 22$ and 
$$
\sum_{k=0}^{4}\PDO(k)\PDO(4-k) = 1\cdot 5 + 1\cdot 4 + 2\cdot 2 + 4\cdot 1 + 5 \cdot 1 = 22
$$ 
which confirms Theorem \ref{id:ori-pdo} when $n=4$.  
\end{example} 


Let $\PDOset(n)$ denote the set of all PDO partitions of $n$, so that $\PDO(n) = |\PDOset(n)|$. Moreover, let $\PDOset:=\bigcup_{n\ge 0}\PDOset(n)$ with $\PDOset(0)$ containing only the empty partition $\varnothing$. 

We then consider the following refinement of Theorem \ref{id:ori-pdo}, which could be viewed as one of the main results of this paper. Let $\ld(\lambda)$ denote the number of different part sizes in the partition $\lambda$. We then have the following:  
\begin{theorem}\label{thm:refine-pdo}
For every $n\ge 0$, we have
\begin{align}\label{id:refine-pdo}
\sum_{\lambda\in\PDOset(2n)} x^{\ld(\lambda)} &= \sum_{k=0}^{n}\left(\sum_{\mu\in\PDOset(k)}x^{\ld(\mu)}\right)\left(\sum_{\nu\in\PDOset(n-k)}x^{\ld(\nu)}\right),
\end{align}
or equivalently, writing 
$$
PDO(x,q):=\sum_{\lambda\in\PDOset}x^{\ld(\lambda)}q^{\abs{\lambda}},
$$
we have for all $n\ge 0$,
\begin{align}\label{id:refine-pdo2}
[q^{2n}]PDO(x,q) &= [q^n]\left(PDO(x,q)\right)^2.
\end{align}
\end{theorem}
An example of this result is worth sharing here.  

\begin{example}
We return to the case when $n=4$.  Of the 22 partitions into odd parts with designated summands which are counted by $\PDO(8)$, we see that there are 8 which contain exactly one part size (all of which are constructed from the ordinary partition $1+1+1+1+1+1+1+1$) and there are 14 which contain exactly two different part sizes (these are constructed from designating the parts in $7+1$, $5+3$, $5+1+1+1$, $3+3+1+1$, and $3+1+1+1+1+1$).  

If we denote by $(\mu | \nu)$ the pairs of $\PDO$ partitions where $\mu$ has weight $k$ and $\nu$ has weight $4-k$, then we see that we can naturally break these into two subsets.  First, the 8 pairs of $\PDO$ partitions with only one designated part are the following:  
\begin{gather*}
    (\varnothing | 1'+1+1+1), \ \ \  (\varnothing | 1+1'+1+1), \ \ \  (\varnothing | 1+1+1'+1), \ \ \  (\varnothing | 1+1+1+1'),\\
(1'+1+1+1 | \varnothing), \ \ \  (1+1'+1+1 | \varnothing), \ \ \ (1+1+1'+1 | \varnothing), \ \ \ (1+1+1+1' | \varnothing). 
\end{gather*}
Moreover, the 14 pairs of $\PDO$ partitions with exactly two designated parts are the following: 
\begin{gather*}
    (\varnothing | 3'+1'), \ \ \  (3'+1' | \varnothing), \ \ \  (1' | 3'), \ \ \ (3' | 1'), \\
(1'+1 | 1'+1), \ \ \  (1'+1 | 1+1'),\ \ \ (1+1' | 1'+1), \ \ \  (1+1' | 1+1'),\\
 \ \ \ (1' | 1'+1+1 ), \ \ \ (1' | 1+1'+1 ), \ \ \ (1' | 1+1+1' ), \\
 (1'+1+1 | 1'), \ \ \ (1+1'+1 | 1'), \ \ \ (1+1+1' | 1'). 
\end{gather*}
\end{example}


\section{Our Refinement and Necessary Mathematical Tools}\label{sec:tools}
We begin our journey towards a combinatorial proof of a refined version of Theorem \ref{id:ori-pdo} by considering the work of Andrews and Rose.  In \cite{AR13}, Andrews and Rose introduced the function 
\begin{align}\label{defn:Gxq}
G(x,q) &:= 1+2\sum_{n\geq 1} T_{2n}(x/2)q^{n^2}
\end{align}
where, for $n\geq 0,$ $T_n(x)$ is the $n^{th}$ Chebyshev polynomial (of the first kind) \cite[p.~101]{AARsf} and is defined by 
$$T_n(\cos\theta) = \cos(n\theta).$$
We note, in passing, that 
the sum-product identity 
\begin{align}\label{id:cos-sum-product}
\cos(A+B)+\cos(A-B)=2\cos(A)\cos(B)
\end{align}
extends naturally to these Chebyshev polynomials in the following sense.
\begin{proposition}[Andrews, Askey, and Roy \cite{AARsf}, (5.1.6)]
For $0\le m\le n$, we have
\begin{align}\label{id:cheby-sum-prod}
T_{n+m}(x)+T_{n-m}(x) &= 2T_m(x)T_n(x).
\end{align}
\end{proposition}
\begin{proof}
Seeing that \eqref{id:cheby-sum-prod} is an identity between two polynomials in $\bR[x]$ with degree $n+m$ (recall that $T_n(x)$ is a polynomial in $x$ of degree $n$), it suffices to show that it holds for more than $n+m$ different values of $x$. Indeed, evaluating both sides of \eqref{id:cheby-sum-prod} at $x=\cos\theta,~0\le \theta\le \pi$ yields precisely the aforementioned sum-product identity satisfied by the cosine function. Alternatively, \eqref{id:cheby-sum-prod} follows from induction on the pair $(m,n)$ and the well-known recurrence relation for $T_n(x)$. Namely, for $n\ge 2$, we have
$$T_n(x)=2xT_{n-1}(x)-T_{n-2}(x).$$
\end{proof}

Andrews and Rose \cite{AR13} connected $G(x,q)$ with a particular variant of the sum-of-divisors function in the following sense. Consider $c_{n,k}:=\sum s_1\cdots s_k$ where the sum is taken over all partitions of $n$ of the form
$$n=s_1(2m_1-1)+\cdots+s_k(2m_k-1)$$
with $0<m_1<\cdots<m_k$. Note that for $k=1$, this is simply the sum over all divisors $j$ of $n$ such that $n/j$ is an odd number. For $k\ge 1$, denote by $C_k(q)=\sum_{n\ge 1}c_{n,k}q^n$ the generating function for $\set{c_{n,k}}_{n\ge 1}$ and assume the convention that $C_0(q):=1$. This family of functions $C_k(q)$ originated in the work of MacMahon \cite{Mac} and has recently received a great deal of attention from Bachmann \cite{Ba24}, Ono and Singh \cite{OS24}, and Amdeberhan, Ono, and Singh \cite{AOS24}.  

We now make the fundamental connection necessary to assist in our work below.  First, note that for a given $k$, $C_k(q)$ provides a natural refinement of the generating function of $\PDO(n)$ in the sense that the coefficient of $q^n$ in $C_k(q)$ provides the contribution to $\PDO(n)$ where the partitions in question have exactly $k$ different part sizes.  Said from a generating function perspective, 
$$
\sum_{n\geq 0} \PDO(n)q^n = \sum_{k\geq 0} C_k(q).
$$
Secondly, thanks to the work of Andrews and Rose \cite{AR13}, we also know the following:
\begin{theorem}[Andrews and Rose \cite{AR13}]
\begin{align}\label{id:G=Ck}
\frac{f_2}{f_1^2}G(x,q) &= \sum_{k\ge 0}C_k(q)x^{2k}.
\end{align}
\end{theorem}
As an aside, we note that $\displaystyle{\frac{f_2}{f_1^2}}$ has been studied extensively in the past since 
$$ 
\sum_{n\ge 0}\overline{p}(n)q^n =\frac{f_2}{f_1^2} 
$$ 
where $\overline{p}(n)$ counts the number of overpartitions of $n$ \cite{CL04}.  
 
The above comments imply that we can now define a ``refined'' version of $\PDO(n)$ which simultaneously takes into account the weight $n$ of the partitions in question as well as the number of different part sizes in the partitions. This prompted us to discover Theorem~\ref{thm:refine-pdo}.

In order to prove (\ref{id:refine-pdo2}), we will need the 2--dissections of $G(x,q)$ as well as $\displaystyle{\frac{f_2}{f_1^2}}$.  Thankfully, these are very straightforward.  First, it is clear from (\ref{defn:Gxq}) that the function $G(x,q)$ can be $2$-dissected easily. 
\begin{theorem}
\label{id:Gxq-2-dis}
We have 
\begin{equation*}
G(x,q) = 1+2\sum_{n\geq 1} T_{4n}\left({x}/{2}\right)q^{4n^2}+2q\sum_{n\geq 1} T_{4n-2}\left({x}/{2}\right)q^{4n^2-4n}.
\end{equation*}
\end{theorem}
Next, thanks to previous work involving $\overline{p}(n)$, we have the following $2$-dissection of the generating function for the overpartition function.
\begin{theorem}[Hirschhorn and Sellers \cite{HS05}]
\label{id:overpar-2-dis}
We have 
\begin{equation*}
\sum_{n\ge 0}\overline{p}(n)q^n =\frac{f_2}{f_1^2} =\frac{f_8^5}{f_2^4 f_{16}^2}+2q\frac{f_4^2 f_{16}^2}{f_2^4 f_8}.
\end{equation*}
\end{theorem}

Combining the fact that 
$$PDO(x,q)=\sum_{k\ge 0}C_k(q)x^{k}$$
with \eqref{id:G=Ck} reveals that \eqref{id:refine-pdo2} is equivalent to the following identity: 
\begin{align}\label{id:Gx}
[q^{2n}]\left(\frac{f_2}{f_1^2}G(x,q)\right) &= [q^n]\left(\frac{f_2}{f_1^2}G(x,q)\right)^2.
\end{align}
In view of Theorem \ref{id:Gxq-2-dis} and Theorem \ref{id:overpar-2-dis}, this is equivalent to
\begin{align*}
\left.\left(\frac{f_8^5}{f_2^4f_{16}^2}\left(1+2\sum_{n\ge 1}T_{4n}q^{4n^2}\right)+4q^2\frac{f_4^2f_{16}^2}{f_2^4f_8}\sum_{n\ge 1}T_{4n-2}q^{4n^2-4n}\right)\right\vert_{q^2\to q}=\frac{f_2^2}{f_1^4}\left(1+2\sum_{n\ge 1}T_{2n}q^{n^2}\right)^2,
\end{align*}
which simplifies to 
\begin{align}\label{id:root}
\frac{f_4^5}{f_2^2f_8^2}\left(1+2\sum_{n\ge 1}T_{4n}q^{2n^2}\right) + 4q\frac{f_8^2}{f_4}\left(\sum_{n\ge 1}T_{4n-2}q^{2n^2-2n}\right) &= \left(1+2\sum_{n\ge 1}T_{2n}q^{n^2}\right)^2,
\end{align}
where we have suppressed the argument $x/2$ in all Chebyshev polynomials $T_n(x/2)$, as this will be tacitly understood in the situation when all Chebyshev polynomials occurring in an expression have the same argument. The two eta-quotients\footnote{Recall the Dedekind eta function $\eta(\tau)=e^{\frac{\pi i\tau}{12}}\prod_{n=1}^{\infty}(1-e^{2n\pi i\tau})=q^{\frac{1}{24}}f_1$ wherein $q=e^{2\pi i\tau}$ with $\mathrm{Im}(\tau)>0$, so we customarily refer to an expression like $\frac{\prod_{i=1}^uf^{r_i}_{\alpha_i}}{\prod_{i=1}^tf_{\gamma_i}^{s_i}}$ as an eta quotient, where $\alpha_i$ and $\gamma_i$ are positive and distinct integers, $r_i,s_i>0$, and $u,t\ge 0$.} appearing on the left-hand-side of \eqref{id:root} are effectively the product representations of the two classical theta functions studied by Gauss, Jacobi and Ramanujan \cite[(1.5.6) and (1.5.7)]{Hir}:
\begin{align*}
\varphi(q) &:= 1+2\sum_{n\ge 1}q^{n^2}=\frac{f_2^5}{f_1^2f_4^2},\\
\psi(q) &:= \sum_{n\ge 1}q^{\binom{n}{2}}=\frac{f_2^2}{f_1}.
\end{align*}

Making necessary changes of variables (i.e., $q\to q^2$ in $\varphi(q)$ and $q\to q^4$ in $\psi(q)$) and substituting the above two expressions into \eqref{id:root}, we see that it now suffices to show the following lemma, which is essentially \eqref{id:root} after another round of $2$-dissection.

\begin{lemma}\label{lem:main 2-dis}
The following identities hold:
\begin{align*}
\left(1+2\sum_{n\ge 1}q^{2n^2}\right)\left(1+2\sum_{n\ge 1}T_{4n}q^{2n^2}\right) &= \left(1+2\sum_{n\ge 1}T_{4n}q^{4n^2}\right)^2+4q^2\left(\sum_{n\ge 1}T_{4n-2}q^{4n(n-1)}\right)^2,\\
\left(\sum_{n\ge 1}q^{2n(n-1)}\right)\left(\sum_{n\ge 1}T_{4n-2}q^{2n(n-1)}\right) &= \left(1+2\sum_{n\ge 1}T_{4n}q^{4n^2}\right)\left(\sum_{n\ge 1}T_{4n-2}q^{4n(n-1)}\right).
\end{align*}
\end{lemma}

We have now worked through the steps necessary to show that Theorem \ref{thm:refine-pdo} holds once we prove Lemma \ref{lem:main 2-dis}. As it turns out, a result more general than Lemma \ref{lem:main 2-dis} actually holds true; see Theorem~\ref{thm:GxGy} and Lemma~\ref{lem:Cheby-xyuv}. Furthermore, Theorem \ref{thm:GxGy} also enjoys a partition theoretical interpretation in terms of refined counting of PDO partitions with respect to new partition statistics; see Theorem \ref{thm:xy-refine}, wherein \eqref{id:P1=P2} generalizes \eqref{id:refine-pdo2} with one more parameter $y$. We state and prove these results in the next two  sections.

\section{A symmetric generalization and a further refinement}

We begin with an observation that is quite useful for $2$-dissection. For any $q$-series $F:=F(q)$, which is allowed to  contain parameters like $x,y,\ldots$, let $\cH$ denote the $2$-dissection operator defined by $\cH\left( \sum_{n\geq 0} a_nq^n\right) = \sum_{n\geq 0}a_{2n}q^n$ (cf. the ``huffing'' operator in~\cite[Sect.~6.3]{Hir}).

\begin{lemma}\label{lem:gen lemma}
Suppose $A(q)$, $B(q)$, and $C(q)$ are any $q$-series. Then we have
\begin{align*}
\cH(A(q^2)B(q))=A(q)\cH(B(q)).
\end{align*} 
In particular, when $A(q)$ is in the form of an eta-quotient, then so is $A(q)^{-1}$, and consequently in this case, $\cH(B(q))=C(q)$ if and only if $\cH(A(q^2)B(q))=A(q)C(q)$.
\end{lemma}
Now we can rewrite \eqref{id:Gx} using the operator $\cH$:
\begin{align*}
\cH\left(\frac{f_2}{f_1^2}G(x,q)\right) &= \frac{f_2^2}{f_1^4}G(x,q)^2.
\end{align*}
Applying Lemma \ref{lem:gen lemma} with $A(q)=f_1^4/f_2^2$ to the above identity, we see that \eqref{id:Gx} is equivalent to 
\begin{align*}
\cH\left(\frac{f_2^5}{f_1^2f_4^2}G(x,q)\right) &= G(x,q)^2,
\end{align*}
or rather
\begin{align}\label{id:G2Gy}
\cH(G(2,q)G(x,q)) &= G(x,q)^2,
\end{align}
after realizing that $G(2,q)=1+2\sum_{n>0}q^{n^2}=\frac{f_2^5}{f_1^2f_4^2}$. With \eqref{id:Gx} reformulated as \eqref{id:G2Gy}, it seems fitting to call the following result a symmetric generalization of \eqref{id:Gx}.

\begin{theorem}\label{thm:GxGy}
Given four variables $x,y,u,v$ satisfying the relation
\begin{equation}\label{correlation}
\begin{cases}
u+v=xy,\\
uv=x^2+y^2-4,
\end{cases}
\end{equation}
we have
\begin{equation}\label{id:GxGy}
\cH(G(x,q)G(y,q))=G(u,q)G(v,q).
\end{equation}
\end{theorem}

We shall prove Theorem~\ref{thm:GxGy} in the next section. Returning to our combinatorial perspective, we proceed to introduce two double parametrized generating functions with respect to two new partition statistics, and then state a further refinement of Theorem \ref{thm:refine-pdo}, which could be viewed as a partition theoretical interpretation of Theorem~\ref{thm:GxGy}.

For a PDO partition $\lambda\in\PDOset$, let $\ldo(\lambda)$ denote  the number of different parts that occur an odd number of times in $\lambda$, and write 
$$P_1(x,y,q):=\sum_{\lambda\in\PDOset}x^{\ld(\lambda)}y^{\ldo(\lambda)}q^{|\lambda|}.$$ 
For instance, we see $\ldo(7'+3+3'+1+1+1+1'+1)=2$ since parts $7$ and $1$ both occur an odd number of times while the part $3$ does not. For a PDO partition pair $(\mu | \nu)$, we use $\lr(\mu,\nu)$ to denote the number of different part sizes that occur simultaneously  in  $\mu$ and $\nu$. We also introduce $$P_2(x,y,q):=\sum_{(\mu | \nu)\in\PDOset\times\PDOset}x^{\ld(\mu)+\ld(\nu)}y^{2\lr(\mu,\nu)}q^{|\mu|+|\nu|}.$$ 
As an example, we have $\lr(3'+3+1+1+1'| 5'+5+3')=1$ since only the part $3$ can be found in both $\mu$ and $\nu$. 

Both generating functions $P_1(x,y,q)$ and $P_2(x,y,q)$ factor nicely.

\begin{theorem}\label{thm:expression-P1-P2}
We have the following: 
\begin{align}
\label{eq:P1}
P_1(x,y,q) &=\prod_{m\ge 1}\left(1+2x\frac{q^{4m-2}}{(1-q^{4m-2})^2}+xy\frac{q^{2m-1}(1+q^{4m-2})}{(1-q^{4m-2})^2}\right),\\
P_2(x,y,q) &=\prod_{m\ge 1}\left(1+2x\frac{q^{2m-1}}{(1-q^{2m-1})^2}+x^2y^2\left(\frac{q^{2m-1}}{(1-q^{2m-1})^2}\right)^2\right).
\label{eq:P2}
\end{align}
\end{theorem}

Before a proof of Theorem~\ref{thm:expression-P1-P2} is given, we state a further refinement of Theorem~\ref{thm:refine-pdo}, as both a partition theorem and an identity between coefficients of two generating functions.
\begin{theorem}\label{thm:xy-refine}
For all $n\ge 0$, and $k\ge j\ge 0$, there are as many PDO partitions of weight $2n$ with $k$ different part sizes and $j$ part sizes that occur an odd number of times, as there are PDO partition pairs of combined weight $n$, with $k$ designated parts, and $j/2$ simultaneously shared part sizes. Restated in terms of generating functions, we have
\begin{align}\label{id:P1=P2}
[q^{2n}]P_1(x,y,q) = [q^n]P_2(x,y,q).
\end{align}
\end{theorem}

We next provide a concrete example for this result.

\begin{example}
For the case $n=4$ and $k=j=2$, we see there are 10 PDO partitions of weight $8$ which contain exactly two different part sizes, and each of them occur an odd number of times (these are derived from designating the parts in $7+1$, $5+3$, $5+1+1+1$, and $3+1+1+1+1+1$). On the other hand, the following 10 (the same count as claimed by Theorem \ref{thm:xy-refine}) PDO partition pairs are all of those having combined weight $4$, with $2$ designated parts and $j/2=1$ shared part size.
\begin{gather*}
(1'+1 | 1'+1), \ \ \  (1'+1 | 1+1'),\ \ \ (1+1' | 1'+1), \ \ \  (1+1' | 1+1'),\\
 \ \ \ (1' | 1'+1+1 ), \ \ \ (1' | 1+1'+1 ), \ \ \ (1' | 1+1+1' ), \\
 (1'+1+1 | 1'), \ \ \ (1+1'+1 | 1'), \ \ \ (1+1+1' | 1'). 
\end{gather*}
\end{example}

For the rest of this section, we first prove Theorem~\ref{thm:expression-P1-P2}, then we continue to explain the equivalence between Theorem~\ref{thm:GxGy} and Theorem \ref{thm:xy-refine}, relying on the two expressions shown in Theorem \ref{thm:expression-P1-P2}.

\begin{proof}[Proof of Theorem~\ref{thm:expression-P1-P2}]
We begin with the derivation of \eqref{eq:P1}.
\begin{align*}
P_1(x,y,q) &= \sum_{\lambda\in\PDOset}x^{\ell_d(\lambda)}y^{\ell_d^o(\lambda)}q^{\abs{\lambda}}\\
&= \prod_{m\ge 1}(1+xyq^{2m-1}+2xq^{2(2m-1)}+3xyq^{3(2m-1)}+4xq^{4(2m-1)}+\cdots)\\
&= \prod_{m\ge 1}(1+x(2q^{2(2m-1)}+4q^{4(2m-1)}+\cdots)+xy(q^{2m-1}+3q^{3(2m-1)}+\cdots))\\
&= \prod_{m\ge 1}\left(1+x\frac{2q^{2(2m-1)}}{(1-q^{2(2m-1)})^2}+xy\frac{q^{2m-1}(1+q^{2(2m-1)})}{(1-q^{2(2m-1)})^2}\right).
\end{align*}
In order to make a similar calculation that yields \eqref{eq:P2}, we take any PDO partition pair $(\mu | \nu)$, pick out one of its parts, say $2m-1$, and consider the following two cases.
\begin{description}
    \item[Case I] If $2m-1$ only occurs in one of $\mu$ or $\nu$, and the number of its copies is, say $t$, then collectively they contribute the weight $2txq^{t(2m-1)}$ to $P_2(x,y,q)$, where the factor $2$ indicates the two possibilities that $2m-1$ may occur in either $\mu$ or $\nu$, and $t$ accounts for the $t$ different ways to designate one part of size $2m-1$.
    \item[Case II] If $2m-1$ occurs in both of $\mu$ and $\nu$ with a total number of occurrences being, say $t$, then the weight contribution is $(\frac{t^3-t}{6})x^2y^2q^{t(2m-1)}$, where the factor $\frac{t^3-t}{6}$ can be justified as follows. Suppose $i$ out of $t$ copies of $2m-1$ come from $\mu$ with $1\le i\le t-1$, then there are $i$ ways to designate one part of size $2m-1$ in $\mu$ and $t-i$ ways to designate one part of size $2m-1$ in $\nu$. Therefore the total possibilities are given by
    $$\sum_{i=1}^{t-1}i(t-i)=t\binom{t}{2}-\frac{t(t-1)(2t-1)}{6}=\frac{t^3-t}{6}.$$ 
\end{description}
The discussion above gives rise to the following calculation.
\begin{align*}
P_2(x,y,q) &= \sum_{(\mu | \nu)\in\PDOset\times\PDOset}x^{\ld(\mu)+\ld(\nu)}y^{2\lr(\mu,\nu)}q^{|\mu|+|\nu|}\\
&= \prod_{m\ge 1}\left(1+x\sum_{t\ge 1}2tq^{t(2m-1)}+x^2y^2\sum_{t\ge 1}\frac{t^3-t}{6}q^{t(2m-1)}\right)\\
&= \prod_{m\ge 1}\left(1+2x\frac{q^{2m-1}}{(1-q^{2m-1})^2}+x^2y^2\frac{q^{2(2m-1)}}{(1-q^{2m-1})^4}\right).
\end{align*}
\end{proof}

Next, we show how to deduce \eqref{id:P1=P2} from \eqref{id:GxGy}. The idea is to fully exploit Andrews-Rose's identity \eqref{id:G=Ck}, Lemma~\ref{lem:gen lemma}, and make an appropriate change of variables. Let us first rewrite \eqref{id:G=Ck} in a more explicit way:
\begin{align}
G(x,q) &= \frac{(q;q)^2_{\infty}}{(q^2;q^2)_\infty}\sum_{k\ge 0}C_k(q)x^{2k}\nonumber \\
&= (q^2;q^2)_{\infty}(q;q^2)^2_{\infty}\prod_{m\ge 1}\left(1+x^2\frac{q^{2m-1}}{(1-q^{2m-1})^2}\right)\nonumber \\
&= (q^2;q^2)_{\infty}\prod_{m\ge 1}\left((1-q^{2m-1})^2+x^2q^{2m-1}\right).
\label{id:G=Ck new}
\end{align}
We pause here to make a key observation. After rewriting the $m$-th factor in \eqref{eq:P1} as a single fraction, we see the numerator is given by $$1+xyq^{2m-1}+(2x-2)(q^{2m-1})^2+xy(q^{2m-1})^3+(q^{2m-1})^4.$$ 
Viewed as a polynomial in $q^{2m-1}$, its palindromic coefficients make it plausible to attempt the factorization
$$(1+Aq^{2m-1}+(q^{2m-1})^2)(1+Bq^{2m-1}+(q^{2m-1})^2)$$
and apply \eqref{id:G=Ck new} twice to assemble $G(x,y)G(y,q)$. Then, determining the coefficients $A$ and $B$ will guide us to the correct change of varibles.

More precisely, given four variables $w,z,u,v$ satisfying the relation \eqref{correlation}, i.e.,
\begin{align}
\label{cor1}
\begin{cases}
u+v=wz,\\
uv=w^2+z^2-4,
\end{cases}
\end{align}
the following identity can be deduced from applying Lemma~\ref{lem:gen lemma} to \eqref{id:GxGy}.
\begin{align}\label{id:GxGy new}
\cH\left(\frac{1}{(q^2;q^4)^2_{\infty}(q^2;q^2)^2_{\infty}}G(w,q)G(z,q)\right) &= \frac{1}{(q;q^2)^2_{\infty}(q;q)^2_{\infty}}G(u,q)G(v,q).
\end{align}
Plugging in the new expression \eqref{id:G=Ck new} for $G(w,q)$ and $G(z,q)$, we see that
\begin{align*}
\frac{G(w,q)G(z,q)}{(q^2;q^4)^2_{\infty}(q^2;q^2)^2_{\infty}} &= \frac{\prod_{m\ge 1}((1-q^{2m-1})^2+w^2q^{2m-1})((1-q^{2m-1})^2+z^2q^{2m-1})}{(q^2;q^4)^2_{\infty}}\\
&= \frac{\prod_{m\ge 1}\left(1+(w^2-2)q^{2m-1}+q^{4m-2}\right)\left(1+(z^2-2)q^{2m-1}+q^{4m-2}\right)}{(q^2;q^4)^2_{\infty}}\\
&= \frac{\prod_{m\ge 1}\left(1+xyq^{2m-1}+(2x-2)q^{4m-2}+xyq^{6m-3}+q^{8m-4}\right)}{(q^2;q^4)^2_{\infty}}\\
&=\prod_{m\ge 1}\left(1+2x\frac{q^{4m-2}}{(1-q^{4m-2})^2}+xy\frac{q^{2m-1}+q^{6m-3}}{(1-q^{4m-2})^2}\right)= P_1(x,y,q),
\end{align*}
where in the penultimate line we have renamed the variables
\begin{align}\label{cor2}
w^2+z^2-4=uv\to xy, \quad (w^2-2)(z^2-2)+4=u^2+v^2\to 2x.
\end{align}
The right hand side of \eqref{id:GxGy new} can be handled analogously. Plugging in \eqref{id:G=Ck new} for $G(u,q)$ and $G(v,q)$, we obtain 
\begin{align*}
\frac{G(u,q)G(v,q)}{(q;q^2)^2_{\infty}(q;q)^2_{\infty}} &= \frac{\prod_{m\ge 1}((1-q^{2m-1})^2+u^2q^{2m-1})((1-q^{2m-1})^2+v^2q^{2m-1})}{(q;q^2)^4_{\infty}}\\
&= \prod_{m\ge 1}\left(1+(u^2+v^2)\frac{q^{2m-1}}{(1-q^{2m-1})^2}+u^2v^2\frac{q^{4m-2}}{(1-q^{2m-1})^4}\right)\\
&= \prod_{m\ge 1}\left(1+2x\frac{q^{2m-1}}{(1-q^{2m-1})^2}+x^2y^2\frac{q^{4m-2}}{(1-q^{2m-1})^4}\right)=P_2(x,y,q).
\end{align*}
To verify the change of variables in the last line, it suffices to combine the relations \eqref{cor1} and \eqref{cor2}. The calculations above indicate that indeed \eqref{id:GxGy new} implies \eqref{id:P1=P2}, and since every step is reversible, we have actually proven that Theorem~\ref{thm:GxGy} is equivalent to Theorem~\ref{thm:xy-refine}.

\section{A proof of Theorem~\ref{thm:GxGy}}
All that remains is to prove Theorem~\ref{thm:GxGy}. After applying the $\cH$ operator on the product $G(x,q)G(y,q)$ and making a $2$-dissection based on Theorem~\ref{id:Gxq-2-dis}, we see that Theorem~\ref{thm:GxGy} is equivalent to the following result.
\begin{lemma}\label{lem:Cheby-xyuv}
Given four variables $x,y,u,v$ related by \eqref{correlation}, the following identities hold.
\begin{align}
&\left(1+2\sum_{n\ge 1}T_{4n}(\frac{x}{2})q^{2n^2}\right)\left(1+2\sum_{n\ge 1}T_{4n}(\frac{y}{2})q^{2n^2}\right)= \left(1+2\sum_{n\ge 1}T_{4n}(\frac{u}{2})q^{4n^2}\right)\left(1+2\sum_{n\ge 1}T_{4n}(\frac{v}{2})q^{4n^2}\right)\nonumber \\
 &\qquad\qquad\qquad +4q^2\left(\sum_{n\ge 1}T_{4n-2}(\frac{u}{2})q^{4n^2-4n}\right)\left(\sum_{n\ge 1}T_{4n-2}(\frac{v}{2})q^{4n^2-4n}\right),\label{id:Cheby-even}\\
&2\left(\sum_{n\ge 1}T_{4n-2}(\frac{x}{2})q^{2n^2-2n}\right)\left(\sum_{n\ge 1}T_{4n-2}(\frac{y}{2})q^{2n^2-2n}\right) = \left(1+2\sum_{n\ge 1}T_{4n}(\frac{u}{2})q^{4n^2}\right)\left(\sum_{n\ge 1}T_{4n-2}(\frac{v}{2})q^{4n^2-4n}\right)\nonumber \\
&\qquad\qquad\qquad +\left(1+2\sum_{n\ge 1}T_{4n}(\frac{v}{2})q^{4n^2}\right)\left(\sum_{n\ge 1}T_{4n-2}(\frac{u}{2})q^{4n^2-4n}\right).
\label{id:Cheby-odd}
\end{align}
In particular, setting $y=2$ and hence $u=v=x$ by \eqref{correlation}, we recover the two identities in Lemma~\ref{lem:main 2-dis}.
\end{lemma}

Both identities are concerned with Chebyshev polynomials in four correlated variables with bounded degrees, so similar to the proof of \eqref{id:cheby-sum-prod}, it suffices to evaluate them at specific values and verify the resulting identities by repeatedly applying the sum-product identity \eqref{id:cos-sum-product}. Let us set $x=2\cos\alpha$, $y=2\cos\beta$, $u=2\cos\mu$, and $v=2\cos\nu$.  Consequently, \eqref{correlation} now reads as
\begin{align*}
\begin{cases}
\cos\mu+\cos\nu=2\cos\alpha\cos\beta,\\
2\cos\mu\cos\nu=2\cos^2\alpha+2\cos^2\beta-2=\cos2\alpha+\cos2\beta,
\end{cases}
\end{align*}
which, in view of \eqref{id:cos-sum-product}, is equivalent to saying that $\cos\mu=\cos(\alpha+\beta)$ and $\cos\nu=\cos(\alpha-\beta)$. Now since all four variables $\alpha,\beta,\mu,\nu$ only appear as arguments in the cosine function, we shall simply enforce that $\mu=\alpha+\beta$ and $\nu=\alpha-\beta$. With the discussion above in mind, one sees that Lemma~\ref{lem:Cheby-xyuv} is equivalent to the following result concerning cosine functions. 


\begin{lemma}\label{lem:new 2-dis}
For a pair of real numbers $\alpha$ and $\beta$, let $\mu:=\alpha+\beta$ and $\nu:=\alpha-\beta$. The following identities hold:
\begin{align}
&\quad \left(1+2\sum_{n>0}\cos(4n\alpha)q^{2n^2}\right)\left(1+2\sum_{n>0}\cos(4n\beta)q^{2n^2}\right)\nonumber\\
&=\left(1+2\sum_{n>0}\cos(4n\mu)q^{4n^2}\right)\left(1+2\sum_{n>0}\cos(4n\nu)q^{4n^2}\right)\nonumber\\
&\quad +4\left(\sum_{n>0}\cos((4n-2)\mu)q^{(2n-1)^2}\right)\left(\sum_{n>0}\cos((4n-2)\nu)q^{(2n-1)^2}\right)\label{id:new 2-dis-e},  \textrm{\ \ and} \\
&\quad 2\left(\sum_{n>0}\cos((4n-2)\alpha)q^{2n^2-2n}\right)\left(\sum_{n>0}\cos((4n-2)\beta)q^{2n^2-2n}\right)\nonumber\\
&=\left(1+2\sum_{n>0}\cos(4n\mu)q^{4n^2}\right)\left(\sum_{n>0}\cos((4n-2)\nu)q^{4n^2-4n}\right) \nonumber\\
&\quad +\left(1+2\sum_{n>0}\cos(4n\nu)q^{4n^2}\right)\left(\sum_{n>0}\cos((4n-2)\mu)q^{4n^2-4n}\right).
\label{id:new 2-dis-o}
\end{align}
\end{lemma}

Before proving Lemma \ref{lem:new 2-dis}, we note the following.
Let $\bP$ be the set of positive integers. We introduce the sets $S:=\set{(n,m)\in\bP\times\bP: n>m}$, and 
$$S_e:=\set{(n,m)\in S: n\equiv m\pmod 2},\quad S_o:=\set{(n,m)\in S: n\not\equiv m\pmod 2}.$$
Note that $S=S_e\bigcup S_o$ and we observe the following simple fact.
\begin{proposition}\label{prop:two bij}
The following two mappings are bijections:
\begin{align*}
f_e:\quad S &\to S_e \\
(n,m) &\mapsto \left(n+m,n-m\right),\\
\noalign{\medskip}
f_o:\quad S &\to S_o \\
(n,m) &\mapsto \left(n+m-1,n-m\right).
\end{align*}
\end{proposition}
We now proceed to proving the two identities that appear in Lemma \ref{lem:new 2-dis}.

\begin{proof}[Proof of \eqref{id:new 2-dis-e}]
We begin by expanding the left hand side of  \eqref{id:new 2-dis-e} as follows.

\begin{align*}
&\quad \left(1+2\sum_{n>0}\cos(4n\alpha)q^{2n^2}\right)\left(1+2\sum_{n>0}\cos(4n\beta)q^{2n^2}\right)\nonumber\\ 
&= 1+2\sum_{n>0}(\cos(4n\alpha)+\cos(4n\beta))q^{2n^2}+4\left(\sum_{n>0}\cos(4n\alpha)q^{2n^2}\right)\left(\sum_{n>0}\cos(4n\beta)q^{2n^2}\right)\\
&\stackrel{\eqref{id:cos-sum-product}}{=} 1+\underbrace{4\sum_{n>0}\cos(2n\mu)\cos(2n\nu)q^{2n^2}}_{A}+\underbrace{4\left(\sum_{n>0}\cos(4n\alpha)q^{2n^2}\right)\left(\sum_{m>0}\cos(4m\beta)q^{2m^2}\right)}_{B}.
\end{align*}
Next we split $A$ according to the parity of the index $n$.
\begin{align*}
A &= \underbrace{4\sum_{n>0}\cos(4n\mu)\cos(4n\nu)q^{8n^2}}_{A_0}+\underbrace{4\sum_{n>0}\cos((4n-2)\mu)\cos((4n-2)\nu)q^{2(2n-1)^2}}_{A_1}.
\end{align*}

Likewise, we split $B$ based on whether $n=m$, $n>m$, or $n<m$.
\begin{align*}
B &= 4\sum_{n>0}\cos(4n\alpha)\cos(4n\beta)q^{4n^2}+4\sum_{(n,m)\in S}\cos(4n\alpha)\cos(4m\beta)q^{2n^2+2m^2}\\
&\quad +4\sum_{(n,m)\in S}\cos(4m\alpha)\cos(4n\beta)q^{2n^2+2m^2}\\
&=B_0+B_1+B_2.
\end{align*}
Equipped with Proposition~\ref{prop:two bij}, we can further transform $B_1$ and $B_2$ via the bijection $f_e$.
\begin{align*}
B_1 &\stackrel{\eqref{id:cos-sum-product}}{=} 2\sum_{(n,m)\in S}(\cos(4n\alpha+4m\beta)+\cos(4n\alpha-4m\beta))q^{2n^2+2m^2}\\
&= 2\sum_{(n,m)\in S}(\cos(2(n+m)\mu+2(n-m)\nu)+\cos(2(n-m)\mu+2(n+m)\nu))q^{(n+m)^2+(n-m)^2}\\
&= \underbrace{2\sum_{(n,m)\in S_e}\cos(2n\mu+2m\nu)q^{n^2+m^2}}_{B_{11}}+\underbrace{2\sum_{(n,m)\in S_e}\cos(2m\mu+2n\nu)q^{n^2+m^2}}_{B_{12}},
\end{align*}

\begin{align*}
B_2 &\stackrel{\eqref{id:cos-sum-product}}{=} 2\sum_{(n,m)\in S}(\cos(4m\alpha+4n\beta)+\cos(4m\alpha-4n\beta))q^{2n^2+2m^2}\\
&= 2\sum_{(n,m)\in S}(\cos(2(n+m)\mu-2(n-m)\nu)+\cos(2(n+m)\nu-2(n-m)\mu))q^{(n+m)^2+(n-m)^2}\\
&= \underbrace{2\sum_{(n,m)\in S_e}\cos(2n\mu-2m\nu)q^{n^2+m^2}}_{B_{21}}+\underbrace{2\sum_{(n,m)\in S_e}\cos(2n\nu-2m\mu)q^{n^2+m^2}}_{B_{22}}.
\end{align*}
We regroup these four pieces in order to apply \eqref{id:cos-sum-product} again.
\begin{align*}
B_1+B_2 &= (B_{11}+B_{21})+(B_{12}+B_{22})\\
&= 2\sum_{(n,m)\in S_e}(\cos(2n\mu+2m\nu)+\cos(2n\mu-2m\nu))q^{n^2+m^2}\\
&\qquad +2\sum_{(n,m)\in S_e}(\cos(2n\nu+2m\mu)+\cos(2n\nu-2m\mu))q^{n^2+m^2}\\
&\stackrel{\eqref{id:cos-sum-product}}{=}4\sum_{(n,m)\in S_e}\cos(2n\mu)\cos(2m\nu)q^{n^2+m^2}+4\sum_{(n,m)\in S_e}\cos(2n\nu)\cos(2m\mu)q^{n^2+m^2}\\
&=4\sum_{n\neq m,~n\equiv m\pmod 2}\cos(2n\mu)\cos(2m\nu)q^{n^2+m^2}.
\end{align*}
Meanwhile, the right hand side of \eqref{id:new 2-dis-e} expands as
\begin{align*}
&\qquad 1+2\sum_{n>0}(\cos(4n\mu)+\cos(4n\nu))q^{4n^2}+4\left(\sum_{n>0}\cos(4n\mu)q^{4n^2}\right)\left(\sum_{m>0}\cos(4m\nu)q^{4m^2}\right)\\
&\qquad\qquad +4\left(\sum_{n>0}\cos((4n-2)\mu)q^{(2n-1)^2}\right)\left(\sum_{m>0}\cos((4m-2)\nu)q^{(2m-1)^2}\right)\\
&\stackrel{\eqref{id:cos-sum-product}}{=}1+4\sum_{n>0}\cos(4n\alpha)\cos(4n\beta)q^{4n^2}\\
&\qquad+4\sum_{n>0}\cos(4n\mu)\cos(4n\nu)q^{8n^2}+4\sum_{n>0}\cos((4n-2)\mu)\cos((4n-2)\nu)q^{2(2n-1)^2}\\
&\qquad +4\sum_{n\neq m}\cos(4n\mu)\cos(4m\nu)q^{4n^2+4m^2}+4\sum_{n\neq m}\cos((4n-2)\mu)\cos((4m-2)\nu)q^{(2n-1)^2+(2m-1)^2}\\
&=1+B_0+A_0+A_1+B_1+B_2,
\end{align*}
which recovers precisely those pieces that constitute the left hand side of \eqref{id:new 2-dis-e} and we are done.
\end{proof}

\begin{proof}[Proof of \eqref{id:new 2-dis-o}]
The essence of the proof is quite similar to that of \eqref{id:new 2-dis-e}, only that we now require the bijection $f_o$ instead of $f_e$ at a certain point. We start with a trichotomy of the left hand side of \eqref{id:new 2-dis-o}:
\begin{align*}
&\quad 2\left(\sum_{n>0}\cos((4n-2)\alpha)q^{2n^2-2n}\right)\left(\sum_{m>0}\cos((4m-2)\beta)q^{2m^2-2m}\right)\\
&= 2\sum_{n>0}\cos((4n-2)\alpha)\cos((4n-2)\beta)q^{4n^2-4n}\\
&\quad +2\sum_{(n,m)\in S}\cos((4n-2)\alpha)\cos((4m-2)\beta)q^{2n(n-1)+2m(m-1)}\\
&\quad +2\sum_{(n,m)\in S}\cos((4m-2)\alpha)\cos((4n-2)\beta)q^{2n(n-1)+2m(m-1)}\\
&:=C+D_1+D_2.
\end{align*}
One applies \eqref{id:cos-sum-product} to deduce that
\begin{align}
C &= \sum_{n>0}(\cos((4n-2)\mu)+\cos((4n-2)\nu))q^{4n^2-4n},\label{id:C}\\
D_1 &= \sum_{(n,m)\in S}(\cos((4n-2)\alpha+(4m-2)\beta)+\cos((4n-2)\alpha-(4m-2)\beta))q^{2n(n-1)+2m(m-1)}\nonumber \\
&= \sum_{(n,m)\in S}\cos(2(n+m-1)\mu+2(n-m)\nu))q^{(n+m-1)^2+(n-m)^2-1} \nonumber \\
&\qquad +\sum_{(n,m)\in S}\cos(2(n+m-1)\nu+2(n-m)\mu)q^{(n+m-1)^2+(n-m)^2-1}\nonumber \\
&\stackrel{f_o}{=}\underbrace{\sum_{(n,m)\in S_o}\cos(2n\mu+2m\nu)q^{n^2+m^2-1}}_{D_{11}}+\underbrace{\sum_{(n,m)\in S_o}\cos(2n\nu+2m\mu)q^{n^2+m^2-1}}_{D_{12}},\nonumber
\end{align}
and, analogously,
\begin{align*}
D_2 &= \underbrace{\sum_{(n,m)\in S_o}\cos(2n\mu-2m\nu)q^{n^2+m^2-1}}_{D_{21}}+\underbrace{\sum_{(n,m)\in S_o}\cos(-2n\nu+2m\mu)q^{n^2+m^2-1}}_{D_{22}}.\nonumber 
\end{align*}
Rearranging the above four pieces, we see that
\begin{align}
D_1+D_2 &= (D_{11}+D_{21})+(D_{12}+D_{22})\nonumber \\
&\stackrel{\eqref{id:cos-sum-product}}{=}2\sum_{(n,m)\in S_o}\cos(2n\mu)\cos(2m\nu)q^{n^2+m^2-1}+2\sum_{(n,m)\in S_o}\cos(2m\mu)\cos(2n\nu)q^{n^2+m^2-1}\nonumber \\
&= 2\sum_{n\not\equiv m\pmod 2}\cos(2n\mu)\cos(2m\nu)q^{n^2+m^2-1}\nonumber \\
&= 2\left(\sum_{n>0}\cos(4n\mu)q^{4n^2}\right)\left(\sum_{m>0}\cos((4m-2)\nu)q^{4m^2-4m}\right)\nonumber \\
&\qquad +2\left(\sum_{m>0}\cos(4m\nu)q^{4m^2}\right)\left(\sum_{n>0}\cos((4n-2)\mu)q^{4n^2-4n}\right)\label{id:D1+D2}.
\end{align}
Adding \eqref{id:D1+D2} back to \eqref{id:C} yields exactly the right hand side of \eqref{id:new 2-dis-o}, as desired.
\end{proof}

\section{Closing Thoughts}

Despite the fact that Theorem~\ref{thm:xy-refine}  provides a two-parameter generalization of \eqref{genfn:ori-pdo}, its pure combinatorial proof is yet to be found. Nonetheless, there is one special case that we would like to bring to the attention of interested reader. This is the case when we set $y=0$ in \eqref{id:P1=P2}. Namely, the number of $PDO$ partitions of $2n$ with $k$ different parts sizes none of which occurs an odd number of times, equals the number of $PDO$ partition pairs $(\mu | \nu)$ of combined weight $n$ such that $\ell_d(\mu)+\ell_d(\nu)=k$ and $\mu$ and $\nu$ have no parts in common. The equinumerosity at this level is transparent enough to be explained bijectively. Given a partition $\lambda\in\PDO(2n)$ with $\ell_d(\lambda)=k$ and $\ell_d^{o}(\lambda)=0$, we initially set $\mu=\nu=\varnothing$. Now for all occurrences (say $2m$) of a certain part size $a$ in $\lambda$, suppose the $i$-th ($1\le i\le 2m$) occurrence has been designated. Then we distribute $m$ copies of $a$ to the partition pair $(\mu |\nu)$ according to the following two cases.
\begin{itemize}
    \item If $1\le i\le m$, then we append all $m$ copies of $a$ to $\mu$, and designate the $i$-th occurrence of $a$ in $\mu$.
    \item If $m+1\le i\le 2m$, then we append all $m$ copies of $a$ to $\nu$, and designate the $(i-m)$-th occurrence of $a$ in $\nu$.
\end{itemize}
We carry out this process for all $k$ part sizes that occur in $\lambda$, and it is easy to see that the final pair $(\mu |\nu)$ satisfies $|\mu|+|\nu|=|\lambda|$, $\ell_d(\mu)+\ell_d(\nu)=\ell_d(\lambda)$, and $\ell_r(\mu,\nu)=0$. Moreover, the correspondence $\lambda\to(\mu |\nu)$ above is clearly reversible and is indeed a bijection. 

We close by noting that, in MacMahon's original treatment~\cite{Mac}, the series $C_k(q)$ was introduced as a variant to
\begin{equation*}
A_k(q):=\sum_{0<m_1<\cdots<m_k}\frac{q^{m_1+\cdots+m_k}}{(1-q^{m_1})^2\cdots(1-q^{m_k})^2},
\end{equation*}
whose partition theoretical interpretation involves $\PD$ partitions. On the other hand, the Chebyshev connection was established in Andrews-Rose's work~\cite{AR13}:
\begin{align}
\label{id:F-A}
F(x,q):=2\sum_{n=0}^{\infty}T_{2n+1}(x/2)q^{n^2+n}=f_2^3\sum_{k\ge 0}A_k(q^2)x^{2k+1}.
\end{align}
With these facts in mind, one may wonder if a similar result could be developed for $F(x,q)$ and $\PD(n)$.

\section*{Acknowledgement}
We are grateful to the anonymous referees for making useful suggestions on revising this paper. The first author was supported by the National Natural Science Foundation of China grants 12171059 and 12371336 and the Mathematical Research Center of Chongqing University.

\end{document}